\documentclass[12pt]{amsart}
\usepackage{fullpage}

\usepackage[colorlinks=true,urlcolor=blue,
bookmarks=true,bookmarksopen=true,citecolor=blue
]{hyperref}

\usepackage{amsmath, amsfonts}      
\usepackage{color} 
\usepackage[all]{xy}
\usepackage{blkarray}
\usepackage{graphicx}

\theoremstyle{plain}
\newtheorem{mainthm}{Theorem}
\newtheorem{thm}{Theorem}[subsection]

\newtheorem{cor}[thm]{Corollary}

\newtheorem{lem}[thm]{Lemma}
\newtheorem{prop}[thm]{Proposition}

\theoremstyle{definition}
\newtheorem{dfn}[thm]{Definition}

\theoremstyle{remark}
\newtheorem{rem}[thm]{Remark}
\newtheorem*{remnonum}{Remark}

\newtheorem*{qsnonum}{Question}

\theoremstyle{plain}

\newcommand{\R}{\mathbb{R}}
\newcommand{\Z}{\mathbb{Z}}
\newcommand{\Q}{\mathbb{Q}}

\newcommand{\C}{\mathbb{C}}
\newcommand{\F}{\mathbb{F}}

\newcommand{\Crit}{\textnormal{Crit\/}}

\newcommand{\Ccal}{\mathcal{C}}
\newcommand{\tor}{\text{Tor}}

\newcommand{\SP}{\mathcal{P}} 

\newcommand{\fcaddress}{francois.charette8@gmail.com}

\begin{document}

\title[Orientable Lagrangian 3-manifolds]{On the cohomology ring of narrow Lagrangian 3-manifolds, quantum Reidemeister torsion, and the Landau-Ginzburg superpotential}
\date{\today}

\author{Fran\c{c}ois Charette}
\thanks{The author was supported by the Institute Mittag-Leffler and by Universit\'e de Montr\'eal.}

\address{Fran\c{c}ois Charette, Universit\'e de Montr\'eal} \email{\fcaddress}

%

\begin{abstract}
  Let $L$ be a  closed, orientable, monotone Lagrangian 3-manifold of a symplectic manifold $(M, \omega)$, for which there exists a local system such that the corresponding Lagrangian quantum homology vanishes.  We show that its cohomology ring satisfies a certain dichotomy, which depends only on the parity of the first Betti number of $L$.  Essentially, the triple cup product on the first cohomology group is shown to be either of maximal rank or identically zero.  This in turn influences the Landau-Ginzburg superpotential of $L$:  either one of its partial derivatives do not vanish on the corresponding local system, or it is globally constant.  We use this to prove that quantum Reidemeister torsion is invariant and can be expressed in terms of open Gromov-Witten invariants of $L$.
\end{abstract}

\maketitle

%
%
\section{Introduction and results}\label{sec:intro}
Given a closed monotone Lagrangian manifold $L$ of a symplectic manifold $(M, \omega)$, one can wonder how restricted its topology can be.  Since every closed manifold is a Lagrangian of its cotangent bundle (the zero-section),  some restrictions must be imposed on $M$ or $L$ in order to get non-trivial results.  Without restricting $M$, one can impose that $L$ is displaceable by some Hamiltonian isotopy, wich implies that its Lagrangian quantum homology $QH(L)$, in the sense of Biran-Cornea \cite{Bi-Co:rigidity}, vanishes.  However, since the work of Cho \cite{Cho:non-unitary}, a milder restriction can be imposed.  Indeed, Lagrangian quantum homology can be twisted by complex one dimensional representations, also called local systems, $\varphi\colon H_1(L;\Z)/\tor \to \C^\times$, and the resulting homology is denoted by $QH^\varphi(L)$.  Then, one considers Lagrangians endowed with a representation such that $QH^\varphi(L)$ vanishes, in which case we say that $L$ is $\varphi$-narrow.  These pairs $(L, \varphi)$ are called trivial objects of the (monotone) Fukaya category of $M$ and are the subject of this paper.  Moreover, by a theorem of Auroux, Kontsevich, Seidel, it is expected that many pairs $(L, \varphi)$ are trivial, see \S \ref{sec:superpotential} and Corollary \ref{cor:narrowexist} for more on this.  Therefore, it is important to study these trivial objects, as they give information on Lagrangians which may also admit non-narrow representations.

For a pair $(L, \varphi)$, Oh \cite{Oh:spectral} introduced a spectral sequence, starting at the singular homology of $L$ and converging to its quantum homology.  Buhovski \cite{Bu:toriaudin} then showed that this is a spectral sequence of algebras, and the algebra structure on the first page is the intersection product.  Hence, a natural question in the present context is then:
\begin{qsnonum}
 Given a trivial object $(L,\varphi)$ in the monotone Fukaya category of $(M, \omega)$, are there restrictions on the homology ring of $L$?
\end{qsnonum}
We focus our attention on closed orientable 3-manifolds.  In this case, there is a 3-form given by the triple intersection product
$$I_R\colon H_2(L;R)\otimes H_2(L;R)\otimes H_2(L;R) \to H_0(L;R)\cong R$$
Here, $R$ is a ring which will be mostly $\Z$ or a field $\F$.  This form was studied by Sullivan in \cite{Sull:3foldring}, who showed that any integral 3-form on a free abelian group of finite rank $b$ can be realized as the intersection form $I_\Z$ of a closed, orientable 3-manifold with second Betti number equal to $b$.

If one considers the symplectic manifold $(M, \omega) = (\C^3, \omega_0 = \sum dx_j\wedge dy_j)$, then there are strong topological restrictions on its orientable Lagrangians, and the corresponding 3-form has maximal rank.
\begin{thm}
\begin{itemize}
 \item (Damian \cite[Theorem 1.5]{Dam:top-mon-Lags}, Evans-K\c{e}dra \cite[Theorem B]{Ev-Ked:monotoneLags})  If $L$ is a closed, orientable and monotone Lagrangian 3-manifold of $(\C^3, \omega_0)$, then $L$ is diffeomorphic to a product $S^1\times \Sigma_g$, where $\Sigma_g$ is an orientable surface of genus $g$.
 \item (Fukaya \cite[Theorem 11.1]{Fu:NATO})  If $L$ is a closed, orientable and prime Lagrangian 3-manifold of $(\C^3, \omega_0)$, then $L$ is diffeomorphic to a product $S^1\times \Sigma_g$.
\end{itemize} 
\end{thm}
In this paper, we consider any monotone 6-dimensional symplectic manifold $(M^6, \omega)$, convex at infinity whenever it is not closed.  $L$ will always be a closed, orientable, monotone, Lagrangian 3-manifold in $(M, \omega)$.  In order to state our result, we first introduce quickly the main characters of our story.  The precise definitions are given further below.
\subsubsection{Representations, Landau-Ginzburg superpotential: see \S \ref{sec:superpotential}}
Let $\hom_0(H_1; \C^\times)$ be the set of complex one-dimensional representations of $H_1(L;\Z)$ that are equal to 1 on the torsion subgroup.  By choosing a basis $\bar{e}_1, \dots, \bar{e}_{b_1(L;\Z)}$ of the free part of $H_1(L;\Z)$ and a corresponding dual basis $z_1, \dots, z_{b_1(L;\Z)}$, $\hom_0(H_1; \C^\times) \cong (\C^\times)^{b_1(L;\Z)}$ is identified with a complex torus of dimension $b_1(L;\Z)$.  On this space, the Landau-Ginzburg superpotential is a holomorphic Laurent polynomial $\SP \colon \hom_0(H_1; \C^\times) \to \C$ which encodes Maslov index two open Gromov-Witten invariants of $(M, L)$ with one marked point.

\subsubsection{Twisted Lagrangian quantum homology and associated invariants: see \S \ref{sec:pearlcomplex}}
Let $\varphi \in \hom_0(H_1; \C^\times)$, and consider a triple $\mathcal{D} = (f, g, J)$ made of a Morse-Smale function $f\colon L \to \R$, a Riemannian metric $g$, and an almost complex structure $J$ compatible with the symplectic form $\omega$.  Then, there is an associated pearl complex $\Ccal^\varphi(\mathcal{D})$ whose homology, $QH^\varphi(L)$, called Lagrangian quantum homology, is independent of the triple $\mathcal{D}$ but depends on $\varphi$.  There is a spectral sequence $E^{*, \varphi}$, first defined by Oh \cite{Oh:spectral}, whose first page is given by $H_*(L;\C)$ and which converges to $QH^\varphi(L)$.  The space of wide representations is defined as $\mathcal{W}_1 = \{ \varphi \in \hom_0(H_1; \C^\times) \; | \; QH^\varphi(L) \cong H_*(L;\C)\}$.  Similarly, the space of narrow representations is $\mathcal{N}_1 = \{ \varphi \in \hom_0(H_1; \C^\times) \; | \; QH^\varphi(L)=0 \}$.  The narrow/wide conjecture states that the complement of $\mathcal{W}_1$ is $\mathcal{N}_1$.  The discriminant of $L$ (Biran-Cornea \cite[\S 4.3]{Bi-Co:lagtop}) is the function 
$$\Delta\colon \mathcal{W}_1 \to \C$$
$$\varphi=(z_1, \dots z_{b_1(L;\Z)}) \mapsto (-1)^{nb_1(L;\Z) + 1}z_1^2 \cdots z_{b_1(L;\Z)}^2 \det \left( \dfrac{\partial^2 \SP}{\partial z_i \partial z_j} (\varphi) \right)$$

\subsubsection{Quantum Reidemeister torsion: see \S \ref{sec:pearltorsion}}
Given a narrow representation $\varphi \in \mathcal{N}_1$, the associated pearl complex $\Ccal^\varphi(\mathcal{D})$ is acyclic, therefore its torsion is defined.  It is a complex number denoted by $\tau_\varphi(L, \mathcal{D})$.  We started the study of this quantity in our previous work \cite{Cha:torsion}, where we proved that it does not depend on generic choice of data $\mathcal{D}$ whenever the spectral sequence $E^{*, \varphi}$ collapses at the second page.  In the current paper, we prove that it is also invariant for orientable 3-manifolds and express it as a certain rational function of open Gromov-Witten invariants.  

\subsubsection{Discs with pointwise constraints and open Gromov-Witten invariants: see \S \ref{sec:torsionandqproduct}}
Given two cycles $x \in H_1(L;\Z)$, $y\in H_2(L;\Z)$ and a homotopy class $B\in \pi_2(M, L)$ of minimal Maslov index two, recall that $GW_{0,2}^B(x,PD(y)) \in \Z$ denotes genus zero open Gromov-Witten invariants, in the homotopy class $B$, with two marked points intersecting the cycles $x$ and $PD(y)$, the Poincar\'e dual of $y$.  Fix also a representation $\varphi \in \hom_0(H_1; \C^\times)$.  These define a linear map 
$$A_\varphi\colon H_1(L;\C) \to H_2(L;\C)$$
$$x \mapsto \sum_{B, \; y \in H_2} GW_{0,2}^B(x,PD(y))\varphi(\partial B) y$$
Similarly, given a fixed Morse function $f_m$, there is a bilinear map
$$Q_\varphi\colon H^{f_m}_2(L;\C)\otimes H^{f_m}_2(L;\C) \to \C \cong H^{f_m}_3(L;\C)$$
that is defined roughly by counting the number, weighted by $\varphi$, of pseudoholomorphic discs with three boundary marked points intersecting, in clockwise order, the cycles $w, z$ and $PD([L])$, where $[L]$ is the fundamental class of $L$.
\begin{remnonum}
The actual definition of $Q_\varphi$ will be given in \S \ref{sec:minimal}.  It involves four different Morse functions and is defined by the quantum product on the minimal pearl complex.  We prove in Proposition \ref{prop:detQ} that the determinant of $Q_\varphi$ is independent of these choices, whenever $\varphi$ is a narrow representation such that $E^{*, \varphi}$ collapses at the third page.  Related results were obtained previously for wide representations by Biran-Cornea \cite{Bi-Co:lagtop}.
\end{remnonum}

\subsection{Main result}\label{sec:mainresults}
Let $e_1, \dots e_{b_2(L;\Z)}$ denote a basis of the free part of $H_2(L;\Z)$, such that $\bar{e}_i \in H_1(L;\Z)$ is Poincar\'e dual to it.  Finally, denote by $\tor_{\text{ev}}(L)$ (resp.\ $\tor_{\text{odd}}(L)$) the torsion subgroup of $\oplus_{k} H_{2k}(L;\Z)$ (resp.\ $\oplus_{k} H_{2k+1}(L;\Z)$).  All determinants in the theorem are computed with respect to the basis $e_i, \; \bar{e_j}$ (see also Lemma \ref{lem:zbasisf}).
\begin{mainthm}\label{thm:main}
 Let $L$ be a monotone, closed and orientable Lagrangian 3-manifold in $(M, \omega)$, with minimal Maslov class $N_L=2$.  Suppose that there exists a narrow representation $\varphi = (z_1, \dots, z_{b_1(L;\Z)})$.  There is the following alternative:
\begin{enumerate}
 \item \label{thm:mainodd}$b_1(L;\Z)$ is odd: then $E^{*, \varphi}$ collapses at the second page, there exists $i$ such that $\frac{\partial \SP}{\partial z_i}(\varphi) \neq 0$ and $I_\C(e_i, \cdot, \cdot)$ is symplectic on $H_2(L;\C)/ \langle e_i\rangle$. If moreover $b_1 \geq 3$, then $\mathcal{W}_1 = \Crit \SP$.  Quantum Reidemeister torsion is given by the formula $$\tau_\varphi(L) = \dfrac{|\text{Tor}_{\text{ev}}(L)|}{|\text{Tor}_{\text{odd}}(L)|}\cdot \dfrac{(z_i\frac{\partial \SP}{\partial z_i}(\varphi))^{b_1(L;\Z)-3}}{\det I_\C(e_i, \cdot, \cdot)|_{H_2/ \langle e_i \rangle}} \in \C^\times / \pm 1$$
 \item \label{thm:maineven} $b_1(L;\Z)$ is even: then $E^{*, \varphi}$ collapses at the third page, $I_\C = 0$ and $\SP$ is constant.  If moreover $\mathcal{W}_1 \neq \emptyset$, then the discriminant of $L$ satisfies $\Delta(\mathcal{W}_1) = 0$.
Finally, $A_\varphi$ is an isomorphism, $Q_\varphi$ is symplectic, $\det Q_\varphi$ depends only on $\varphi$ and 
$$\tau_\varphi(L)^{b_1(L;\Z)}= \left(\dfrac{|\tor_{\text{ev}}(L)|}{|\tor_{\text{odd}}(L)|}\right)^{b_1(L;\Z)} \dfrac{\det A_\varphi^{b_1(L;\Z)-1}}{\det Q_\varphi} \in \C^\times / \pm 1$$
\end{enumerate}
\end{mainthm}
\begin{cor}\label{cor:main}
Suppose that $b_1$ is odd and $L$ satisfies the hypotheses of the theorem, then $L$ is rationally prime, i.e.\ if $L$ is a connected sum, $L=A \# B$, then either $A$ or $B$ is a $\Q$-homology sphere.
\end{cor}
Note that if $N_L > 2$ and there exists a narrow representation, then $L$ must be a rational homology sphere and $N_L=4$, see e.g.\ Damian \cite[Theorem 1.4]{Dam:Audinsconj} or Fukaya-Oh-Ohta-Ono \cite{FO3:book-vol1}.   

Except for the superpotential, we prove the theorem over $\C$ as well as other fields, with the restriction that the characteristic is different than two and does not divide the order of the torsion subgroup of $H_*(L;\Z)$.  The proof is done in many different steps, spread through the whole paper, and is summarized in \S \ref{sec:mainproof}.  Let us give a list of the steps and point to the relevant sections.\\

\noindent \textit{Step 1:  The spectral sequence and the homology ring.}  In \S \ref{sec:ring}, we prove that if $E^{*, \varphi}$ collapses at the second page, then $I_\C(e_i, \cdot, \cdot)$ is symplectic on $H_2(L;\C)/ \langle e_i\rangle$, from which we get that $b_1(L;\Z)$ is odd, as well as Corollary \ref{cor:main}.  Furthermore, if $E^{*, \varphi}$ collapses at the third page, we get $I_\C = 0$.  Note that this does not imply, a priori, that $b_1(L;\Z)$ is even.  This is done in Step 3.\\

\noindent \textit{Step 2: The invariance of torsion.}  In \S \ref{sec:torinv}, we prove that quantum Reidemeister torsion does not depend on the choice of generic pearl data triple $\mathcal{D}$.  The most interesting part is when the spectral sequence $E^{*, \varphi}$ collapses at the third page, where we discuss briefly bifurcation analysis.  We also give formulas for $\tau_\varphi(L)$ in terms of the differentials in $E^{*, \varphi}$.\\

\noindent \textit{Step 3: Open Gromov-Witten invariants and quantum product.}  In \S \ref{sec:torsionandqproduct}, we express the formulas from Step 2 in terms of the open Gromov-Witten invariants $A_\varphi$ and the chain-level quantum product $Q_\varphi$.  Moreover, we show that $A_\varphi$ is an isomorphism and $Q_\varphi$ is a symplectic form on $H_2(L;\C)$, provided that $E^{*, \varphi}$ collapses at the third page.  This proves that $b_1(L;\Z)$ must be even in this case.\\

\noindent \textit{Step 4: The superpotential.} Finally, we translate our results in terms of the Landau-Ginzburg superpotential in \S \ref{sec:superpotential}.

\subsection{Examples}\label{sec:examples}
The most interesting conclusion of Theorem \ref{thm:main} concerns Lagrangians with an even first Betti number.  Unfortunately, we have to confess that the only examples of this type we could find are rational homology spheres.  It seems that real parts of Fano 3-folds may provide a source of non trivial examples, see the discussion below.
\subsubsection{Odd first Betti number}
These are rather easy to come by.  Indeed, let $S^1$ be a contractible, monotone, circle in a Riemann surface $(\Sigma, \omega_1)$, possibly not compact.  Let $\Sigma_g$ be a monotone, closed, orientable Lagrangian surface in a symplectic 4-manifold $(M^4, \omega_2)$.  Then the product embedding $S^1\times \Sigma_g \subset (\Sigma \times M, \omega_1 \oplus \omega_2)$ admits a narrow representation and has minimal Maslov class two.

\subsubsection{Even first Betti number}
The Chiang Lagrangian \cite{Chiang:RP3} is a rational homology sphere in $\C P^3$ with minimal Maslov class two.  Its Lagrangian quantum homology was studied extensively by Evans-Lekili \cite{Ev-Lek:ChiangLag} and vanishes over fields of characteristic zero, hence the hypotheses of Theorem \ref{thm:main} are satisfied.  Other Chiang-type examples were further studied by Smith \cite{SmiJ:ChiangtypeLag}.

There are many closed orientable $3$-manifolds $L$ with even first Betti numbers whose cohomology ring satisfies $I_\C = 0$.  For example, any orientable circle bundle with a non-vanishing Euler class, or more generally, Seifert fibered spaces with a non-vanishing Euler class.  These can also be made Lagrangian, albeit not monotone, as real parts of complex 3-folds, see Mangolte \cite[Theorems 2.13 \& 4.6]{Man:Real3folds} for an extensive survey.  However, the methods mentionned there do not seem to work for Fano 3-folds, whose real parts would yield monotone Lagrangians.  One would then need to check that they admit narrow representations, so that the hypotheses of Theorem \ref{thm:main} are satisfied.  We will not pursue this direction here.

\subsection{Acknowledgments}
Most of this research was done during a stay at the Institute Mittag-Leffler in the Fall semester 2015, where I benefited from conversations with many of the participants.  I would like to thank the Institute and the organizers of this special thematic semester for the financial support and for providing such an excellent working environment.
\section{Definitions and setting}\label{sec:setting}
This whole section follows very closely the presentation of similar material in our previous work \cite[\S 2 and 3]{Cha:torsion}.
\subsection{The torsion of a chain complex}\label{sec:torsion}
We adapt Milnor's presentation from \cite{Mil:torsion} to our context.
\subsubsection{Non-acyclic complexes: torsion subgroups}\label{sec:milnorsdef}
Let $0 \to C_n \to C_{n-1} \to \cdots \to C_1 \to C_0 \to 0$ be a bounded chain complex over a field $\mathbb{F}$ such that each $C_i$ has a preferred finite basis $c_i=(c_{i,1}, \dots, c_{i, r_i})$.  Given any vector space $V$, we call two bases $b_1, b_2$ of $V$ equivalent if the change of basis matrix $[b_1/b_2]$ has determinant plus or minus one.

Denote by $B_i$ the image of the boundary morphism $d\colon C_{i+1} \to C_i$, by $Z_{i+1}$ its kernel and the resulting homology by $H_i(C_*) = Z_i/B_i$.  Choose bases $b_i$ of $B_i$ and $h_i$ of $H_i$. Then, there are split exact sequences
\begin{align}\label{eq:exactseqmilnortorsion}
&\xymatrix{0 \ar[r] & Z_i \ar[r] & C_i \ar[r]^d & B_{i-1} \ar[r] & 0}\\
&\xymatrix{0 \ar[r] & B_i \ar[r] & Z_i \ar[r] & H_i \ar[r] & 0} \nonumber
\end{align}
which combine to yield a new basis $b_{i-1} h_i b_{i}$ of $C_i$.  
\begin{dfn}\label{eq:nonacyclictorsion}
The torsion is defined as
$$\tau(C_*, c_*, h_*) = \prod_{i=0}^n \det [h_i b_i b_{i-1}/c_i]^{(-1)^i} \in \mathbb{F}^{\times}/ \pm 1.$$
\end{dfn}
It is independent of the choice of basis $b_i$ and of the splittings of the exact sequences, however it does depend on the equivalence class of the bases $c_*$ and $h_*$, since
\begin{equation}\label{eq:torsionchangebasis}
\tau(C_*, c'_*, h'_*) = \tau(C_*, c_*, h_*)\prod_{i=0}^n\det[c_i/c'_i]\det[h_i/h'_i]
\end{equation}

Now, let $f\colon X \to \R$ be a Morse-Smale function on a closed manifold $X$.  Denote by $C_i$ the Morse complex over $\Z$ generated by critical points of $f$ of index $i$, with a $\Z$-basis $c_i$ given by these critical points.  As above, let $B_i$ and $Z_i$ be the groups of boundaries and cycles, which are free $\Z$-modules, say $B_i \cong \Z^{k(i)}$ and $Z_i \cong \Z^{r(i)}$.  Write $H_i(L;\Z) = H_i^{\text{free}}\oplus \tor(H_i(L;\Z)) = \Z^{r(i) - k(i)}\oplus \Z/ a_1 \Z \oplus \cdots \oplus \Z/ a_{s(i)} \Z$.  By standard algebra for modules over principal ideal domains, one can choose bases of $Z_i$ and $B_i$ such that, in the second exact sequence of (\ref{eq:exactseqmilnortorsion}), we have
\begin{align}\label{eq:trivialext}
\Z^{k(i)}=B_i & \to Z_i=\Z^{r(i)} \nonumber\\ 
b_l & \mapsto 
\begin{cases}
a_l z_l \quad 1 \leq l \leq s(i)\\
z_l     \quad s(i) < l \leq k(i) 
\end{cases}
\end{align}
In other words, there is only one free extension (of a given rank) of $H_i(L;\Z)$ by another free module (of a given rank), and, in the appropriate bases, it is given by the obvious maps written above.  Moreover, the relevant change of bases matrices have determinants equal to $\pm 1$, since they give isomorphisms of free $\Z$-modules.

Now, take a field $\F$ whose characteristic does not divide any of the $a_l$'s and tensor both sequences in (\ref{eq:exactseqmilnortorsion}) with $\F$.  This preserves exactness, boundaries, and cycles, by assumption on the characteristic.  The following is rather trivial, it follows for instance by applying the universal coefficient theorem for homology, but it will be useful for computing quantum Reidemeister torsion later on.
\begin{lem}\label{lem:zbasisf}
 Let $\F$ be a field whose characteristic does not divide $|\tor (H_*(L;\Z))|$.  Then a $\Z$-basis $h_*$ of the free part of $H_*(L;\Z)$ gives a basis of $H_*(L;\F)$, by taking $h_*\otimes 1$.  Moreover, any two such bases are equivalent.
\end{lem}
Then, the second sequence (considered over $\F$) in (\ref{eq:exactseqmilnortorsion}) is now free, hence it splits, and is given by
(\ref{eq:trivialext}) in the basis $h_*\otimes 1$.  Therefore, we get $\det [h_i b_i b_{i-1}/c_i] = \prod_{j=1}^{s(i)} a_j = |\tor H_i(L;\Z)| \in \F^\times / \pm 1.$
Finally, 
\begin{equation*}
\tau(C_* \otimes \F, c_* \otimes 1, h_* \otimes 1) = \prod_i |\tor(H_i(L;\Z))|^{-1^i}
\end{equation*}
Simply put, torsion equals torsion!  We will often abbreviate this formula as
\begin{equation}\label{eq:torsioneqtorsion}
\tau(C_* \otimes \F, c_* \otimes 1, h_*\otimes 1) = \dfrac{|\tor_{\text{ev}}(L)|}{|\tor_{\text{odd}}(L)|}
\end{equation}

\subsubsection{Periodic complexes}\label{sec:periodictorsion}
Consider now a 2-periodic chain complex $C_{[*]}$ over a field $\mathbb{F}$:
$$\xymatrix{C_{[1]} \ar@/^1pc/[r]^{d} & \ar@/^1pc/[l]^d C_{[0]}}$$
Assume that $C_{[*]}$ is acyclic, choose bases $c_{[i]}$ of $C_{[i]}$ and $b_{[i]}$ of $B_{[i]} = d(C_{[i-1]})$.
\begin{dfn}
The torsion of a 2-periodic acyclic chain complex is
$$\tau_2(C_{[*]}, c_{[*]} ) = \dfrac{\det[b_{[1]}b_{[0]}/c_{[0]}]}{\det[b_{[0]}b_{[1]}/c_{[1]}]} \in \mathbb{F}^{\times}/\pm 1.$$
\end{dfn}
It is independent of the choice of bases $b_{[i]}$, sections, and equivalence class of $c_{[*]}$.

\subsection{The pearl complex and its torsion}\label{sec:pearlcomplex}
We refer to Biran--Cornea's papers \cite{Bi-Co:Yasha-fest, Bi-Co:rigidity, Bi-Co:lagtop} for foundations and applications of Lagrangian quantum homology.  The version we use here (with oriented moduli spaces of pearls) is adapted from \cite{Bi-Co:lagtop}.

Throughout the text, $(M, \omega)$ is a $6$-dimensional symplectic manifold that is connected and convex at infinity whenever it is not closed.  The space of $\omega$-compatible almost complex structures on $M$ is denoted by $\mathcal{J}_{\omega}$.  All Lagrangian submanifolds $L \subset (M, \omega)$ are closed, connected and orientable 3-manifolds.  Moreover, they are endowed with a fixed choice of orientation and spin structure which we do not write.

Let $\omega\colon \pi_2(M, L) \to \R$ be given by the symplectic area of discs and  $\mu\colon \pi_2(M, L) \to \Z$ denote the Maslov index.  The positive generator of the image of $\mu$ is called the minimal Maslov index and is denoted by $N_L$.  Lagrangians are assumed \textbf{monotone}, that is, there exists a constant $\tau > 0$ such that
\begin{itemize}
	\item $\omega = \tau \mu$
	\item $N_L \geq 2$
\end{itemize}
Since $L$ is orientable, $N_L$ is even.

\subsubsection{The 2-periodic pearl complex}\label{sec:2pearl}
Fix a triple $\mathcal{D}=(f, \rho, J)$, where $\rho$ is a Riemannian metric, $f\colon L \to \R$ a Morse-Smale function and $J \in \mathcal{J}_{\omega}$.  

Set $$\mathcal{C}_k = \mathcal{C}_k(\mathcal{D}) = \Z[\pi_2(M, L)] \langle \text{Crit}_{k}f \rangle, \quad k=0, \dots, n=\dim L,$$ where $\text{Crit}_k f$ is the set of critical points with Morse index $k$ and $\Z[G]$ is the group ring of a group $G$.  We write the Morse index of a critical point $x$ as $|x|$.  For a generic triple $\mathcal{D}$, the pearl differential is defined by
\begin{align*}
d\colon \oplus_k \mathcal{C}_{k}(\mathcal{D}) & \to \oplus_k \mathcal{C}_{k}(\mathcal{D})\\
\text{Crit f} \ni x & \mapsto \sum_{y \in \text{Crit } f}\big(\sum_{\substack{A \in \pi_2(M, L)\\|x|-|y|-1+\mu(A)=0}} \#(\mathcal{P}(x,y,A))A\big)y
\end{align*}
where $\#(\mathcal{P}(x,y,A))$ is the (signed) number of pearls in the homotopy class $A$ going from $x$ to $y$.  When $\mu(A)=0$, a pearl is simply a negative gradient flow line of $f$.  This morphism decomposes as a finite sum 
\begin{equation}\label{eq:dsum}
d = d_M + d_1 + \dots
\end{equation}
 where $d_M\colon \mathcal{C}_k \to \mathcal{C}_{k-1}$ is the Morse differential and $d_i\colon \mathcal{C}_{k} \to \mathcal{C}_{k-1 +iN_L}$ counts pearls of Maslov index $iN_L$.  Note that $d_i$'s are not differentials, they do not square to zero, even though $d^2=0$. 

Since $N_L$ is even, $k$ and $k-1 +i N_L$ have different parity, hence there is a 2-periodic pearl complex over $\Z[\pi_2(M, L)]$, defined by
$$\mathcal{C}_{[*]}(\mathcal{D})=\bigoplus_{k \equiv [*] \; \text{mod } 2} \mathcal{C}_k(\mathcal{D}), \quad [*]=0,1$$
with an induced differential $d\colon \mathcal{C}_{[*]} \to \mathcal{C}_{[*-1]}$.

The homology of this complex is called the quantum homology of $L$, denoted by $QH_{[*]}(L)$, or simply $QH(L)$.  It is independent of generic choices of $\mathcal{D}$.  If $QH(L) =0$, we say that $L$ is \textit{narrow}.

\subsubsection{Narrow representations and torsion}\label{sec:pearltorsion}
Fix a field $\mathbb{F}$ and a ring morphism (by convention, ring morphisms map 1 to 1)
$$\varphi\colon \Z[\pi_2(M, L)] \to \mathbb{F},$$
so that $\mathbb{F}$ becomes a $\Z[\pi_2(M, L)]$-module.  This defines a 2-periodic chain complex over $\mathbb{F}$ by setting
$$\mathcal{C}^{\varphi}_{[*]}(\mathcal{D})=\mathcal{C}_{[*]}(\mathcal{D})\otimes_{\Z[\pi_2(M, L)]} \mathbb{F}, \quad d^{\varphi} = d\otimes 1.$$
As above, the homology of this new complex, denoted by $QH^{\varphi}(L)$, does not depend on $\mathcal{D}$.  If it vanishes, we say that $L$ is $\varphi$-narrow, which implies $\chi(L;\F)=0$.  Notice that $QH(L)=0$ implies $QH^{\varphi}(L)=0$ for every $\varphi$.  

The set of narrow representations of $L$ (see also \S \ref{sec:superpotential}) over $\mathbb{F}$ is defined by
$$\mathcal{N}(L, \mathbb{F})= \left\{ \varphi\colon \Z[\pi_2(M, L)] \to \mathbb{F} \; | \; L \text{ is } \varphi\text{-narrow} \right\}.$$
Given $\varphi \in \mathcal{N}(L, \mathbb{F})$ and $\mathcal{D}$ a generic set of data, there is a preferred basis for $\mathcal{C}^{\varphi}_{[*]}(\mathcal{D})$ given by $\text{Crit}_{[*]} f$.  Proceeding as in \S \ref{sec:periodictorsion}, we have:
\begin{dfn}
The quantum Reidemeister torsion of $L$ is
$$\tau_{\varphi}(L, \mathcal{D}) = \tau_2(\mathcal{C}^{\varphi}_{[*]}(\mathcal{D}), \text{Crit}_{[*]} f) \in \mathbb{F}^{\times}/\pm 1$$
\end{dfn}

\begin{remnonum}
Existence of narrow representations is discussed in \S \ref{sec:superpotential}, see Corollary \ref{cor:narrowexist}.\\
\end{remnonum}

\subsubsection{Oh's spectral sequence}\label{sec:d1complex}
In this section, we will briefly need Novikov ring coefficients in order to define Oh's spectral sequence \cite{Oh:spectral} in Lagrangian Floer homology.  In the pearl context, this corresponds to the degree spectral sequence of Biran--Cornea \cite{Bi-Co:rigidity}.

Set $\Lambda = \Lambda = \F[t, t^{-1}]$ the ring of Laurent polynomials in $t$.  We set $\deg t = |t|=- N_L$.  Set also 
$$P_i =
\begin{cases}
\F t^{-i/N_L} \quad i \equiv 0 \mod N_L\\
0 \quad \text{otherwise}
\end{cases}
$$
Given a generic pearl triple $\mathcal{D} = (f, \rho, J)$, and a representation $\varphi\colon \pi_2(M, L)\to \F^\times$, we have:
\begin{thm}[The degree spectral sequence]\label{thm:specseq}
There is a spectral sequence $\{ E_{p,q}^{r, \varphi}, d^{r, \varphi}; \Lambda \}$, with $d^{r, \varphi}$ of bidegree $(-r, r-1)$, called the degree spectral sequence, which has the following properties:
\begin{itemize}
\item $E_{p,q}^{0, \varphi} = \Ccal_{p+q-pN_L}^\varphi(\mathcal{D}) \otimes P_{pN_L}$, $d^{0, \varphi} = d_M \otimes 1$
\item $E_{p,q}^{1, \varphi} = H_{p+q-pN_L}(L; \F)\otimes P_{pN_L}$, $d^{1, \varphi} = d_{1*}^\varphi \otimes t$, where
$$d_{1*}^\varphi \colon H_r(L;\F) \to H_{r-1+N_L}(L;\F)$$
is induced from the first term $d_1^\varphi$ in the decomposition $d^\varphi = d_M + \sum_i d_i^\varphi$.
\item $d_{1*}^\varphi$ satisfies the Leibniz rule with respect to the Morse intersection product, i.e.\ $d_{1*}^\varphi(x \cdot y) = d_{1*}^\varphi(x)\cdot y +(-1)^{n-|x|}x \cdot d_{1*}^\varphi(y)$
\item $\{ E_{p,q}^{r, \varphi}, d^{r, \varphi}\}$ collapses after at most $\lfloor \frac{n+1}{N_L}\rfloor$ pages.  Moreover, it converges to $QH^\varphi(L; \Lambda)$.  In particular $\oplus_{p+q=l}E_{p,q}^\infty \cong QH_l^\varphi(L;\Lambda)$
\end{itemize}
\end{thm}

\subsubsection{The minimal pearl complex}\label{sec:minimal}
We recall here some properties of the minimal pearl complex, following Biran-Cornea \cite[\S 4.1]{Bi-Co:rigidity}.  Given a generic pearl data triple $\mathcal{D} = (f, \rho, J)$ and its associated twisted pearl complex $(\mathcal{C}^\varphi, d^\varphi)$, there exists a minimal  pearl complex, denoted by $(\mathcal{C}^\varphi_{min}, \delta^\varphi)$, with the following properties.
\begin{itemize}
 \item $\mathcal{C}^\varphi_{min} = (H_*^{f_m}(L;\F)\otimes \Lambda, \delta^\varphi = \sum \delta_i^\varphi), \quad \delta_0=0, \; \delta_i^\varphi\colon H_k^{f_m}(L;\F) \to t^i \cdot H_{k+iN_L-1}^{f_m}(L;\F)$.  Here $f_m$ is a fixed Morse-Smale function and $H_*^{f_m}(L)$ denotes the Morse homology of $L$ with respect to $f_m$.
 \item There exists a quasi-isomorphism $\phi_f\colon \mathcal{C}^\varphi(\mathcal{D}) \to \mathcal{C}^\varphi_{min}$ with quasi inverse denoted by $\psi_f$.  Moreover, $\phi_f$ is compatible with the degree filtrations on both complexes and therefore induces a spectral sequence morphism that is an isomorphism from the first page onwards.  On the first page, it equals the Morse comparison isomorphisms $[\Phi_f^{f_m}]\colon H_*^f(L;\F) \to H_*^{f_m}(L;\F)$.  In other words, there is an associated minimal degree-spectral sequence $(\mathcal{E}^{*, \varphi}_{*,*}, \delta^{*, \varphi})$ and a commutative diagram
 \begin{equation}\label{eq:diagmin}
 \xymatrix{E^{k, \varphi}_{p,q} \ar[r]^{d^{k, \varphi}} \ar[d]^{[\Phi_f^{f_m}]} & E^{k, \varphi}_{p-k,q+k-1} \ar[d]^{[\Phi_f^{f_m}]}\\
 \mathcal{E}^{k, \varphi}_{p,q} \ar[r]^{\delta^{k, \varphi}} &  \mathcal{E}^{k, \varphi}_{p-k,q+k-1}}
 \end{equation}
 \item The quantum product on $(\mathcal{C}^\varphi, d^\varphi)$ induces a quantum product on $\mathcal{C}^\varphi_{min}$.  This is defined as (see \cite[Remark 4.1.3]{Bi-Co:rigidity})
 $$\xymatrix{
 \mathcal{C}^\varphi_{min} \otimes \mathcal{C}^\varphi_{min} \ar[r]^{\psi_{f_1} \otimes \psi_{f_2}} &\mathcal{C}^\varphi(f_1) \otimes \mathcal{C}^\varphi(f_2) \ar[r]^(0.65){\circ} & \mathcal{C}^\varphi(f_3) \ar[r]^{\phi_{f_3}} & \mathcal{C}^\varphi_{min}}$$
 and will also be denoted by $\circ$.  At the chain level, this product depends on the choice of  pearl data associated to $f_m, f_1, f_2$ and $f_3$, which we will not write in order to lighten the notation.  Moreover, it satisfies the Leibniz rule with respect to $\delta^\varphi$ and it is a deformation of the Morse intersection product, meaning that for all $x \in H_p^{f_m}(L), y \in H_q^{f_m}(L)$, we have $x \circ y = x \cdot y + h.o.t.$, where $x \cdot y \in H^{f_m}_{p+q-\dim L}(L)$ is the usual intersection product and h.o.t.\ lives in $\oplus_{j\geq 1} H^{f_m}_{p+q-\dim L + jN_L} \otimes \Lambda$.
 \end{itemize}
 \begin{remnonum}
 It is essential to use field coefficients in order to apply the minimal complex construction, since it makes use of convergence of the degree spectral sequence to the quantum homology.
\end{remnonum}
\section{The cohomology ring of $\varphi$-narrow Lagrangian 3-folds}\label{sec:ring}
Let $\F$ be a field with $\text{char } \F \neq 2$, $L$ a monotone, closed, orientable Lagrangian 3-manifold in $(M, \omega)$, and $I_\F\colon H_2(L; \F) \otimes H_2(L;\F) \otimes H_2(L;\F) \to H_0(L; \F) \cong \F$ the 3-form defined by the triple intersection product.  We will often write $H_i$ instead of $H_i(L; \F)$, to lighten the notation.  Fix a Poincar\'e dual basis $e_i \in H_2, \; i=1, \dots, \dim H_2 =:b_2(L; \F)$ and $\bar{e}_i \in H_1$ , so that $e_i \cdot \bar{e}_j = \delta_{i,j} p$, where $p$ is the generator of $H_0$.  One easily checks that 
\begin{equation}\label{eq:formulaI}
e_i \cdot e_j = \sum_k I_\F(e_i, e_j, e_k) \bar{e}_k.
\end{equation}

\begin{thm}\label{thm:lag3form}
Suppose that there exists a narrow representation $\varphi\colon \pi_2(M, L) \to \F^\times$, where $\text{char } \F \neq 2$.  Then, either $I_\F$ is identically zero or there exists $i$ such that $I_\F(e_i, \cdot, \cdot)$ is a symplectic form on $H_2 / \langle e_i \rangle$.
\end{thm}
\begin{rem}\label{rem:Sullivan}
As observed by Sullivan \cite{Sull:3foldring}, any 3-form on a free abelian group of rank $b$ can be realized as the intersection form of a closed orientable 3-manifold with $b_2(L) = b$.  Moreover, the space of such forms depends on $\frac{1}{6}b(b -1)(b-2)$ parameters, modulo  the action of the group $GL(b)$ of dimension $b^2$.  For related results over $\Z / n\Z$, see the work of Turaev \cite{Tur:cohomring, Tur:cohomringenglish}.
\end{rem}

\begin{proof}
We assume that $\dim H_2 \geq 2$, otherwise there is nothing to prove.  Since $L$ is $\varphi$-narrow, its minimal Maslov number must be $N_L=2$, for degree reasons.  Fix a generic pearl data triple $(f, \rho, J)$ and consider the associated degree spectral sequence $E_{*, *}^{r, \varphi}$ twisted by $\varphi$.  Denote by $v$ the collapsing page, i.e.\ $v = \min \{ r \; | \; E_{*, *}^{r, \varphi}\equiv 0, \; E_{*, *}^{r-1, \varphi} \neq 0\}$.  We have either $v=2$ or $v=3$, as $L$ is $\varphi$-narrow and $\dim L =3$.

\noindent \textbf{Case 1}: $v=3$.  Since the homology of the first page does not vanish, we must have $0=d_{1*}^\varphi\colon H_2 \to H_3$.  For degree reasons, $\dfrac{\ker d_{1*}^\varphi\colon H_2 \to H_3}{\text{im } d_{1*}^\varphi\colon H_1 \to H_2}$ survives to $E^{\infty, \varphi}$, and the latter is null, since $L$ is $\varphi$-narrow, hence the map $A_\varphi:=d_{1*}^\varphi\colon H_1 \to H_2$ is an isomorphism. Finally, by Lemma \ref{lem:d1dual}, $d_{1*}^\varphi(p)=0$.  Since $d_{1*}^\varphi$ is a derivation (see related results of Buhovski \cite{Bu:toriaudin}), we have, for all $i, j$,
$$0=d_{1*}^\varphi(\delta_{i,j} p) = d_{1*}^\varphi(\bar{e}_i \cdot e_j) = d_{1*}^\varphi(\bar{e}_i) \cdot e_j + \bar{e}_i \cdot d_{1*}^\varphi(e_j) = A_\varphi(\bar{e}_i) \cdot e_j$$
The map $A_\varphi$ is an isomorphism, therefore the product vanishes identically on $H_2$ and $I_\F \equiv 0$.  Note that it was not necessary to assume that $\text{char } \F \neq 2$ in this case.

\noindent \textbf{Case 2}:  $v=2$.  Since the homology of the first page is zero, we have the following exact sequence, where the maps are given by $d_{1*}^\varphi$:
$$
\xymatrix{0 \ar[r] & H_0 \ar[r] & H_1 \ar[r] & H_2 \ar[r] & H_3 \ar[r] & 0}
$$
Set $d_{1*}^\varphi(e_i) = r_i L$.  By exactness, we have that $d_{1*}^\varphi\colon H_2 \to H_3$ is surjective, and we can assume without loss of generality that $r_1 \neq 0$.  Then a basis for the kernel of $d_{1*}^\varphi$ is given by $B:= \{  e_i - \frac{r_i}{r_1}e_1 \; | \; i \geq 2 \}$, which is not empty since we assume $\dim H_2 \geq 2$.  Again by exactness, this gives a basis for the image of $d_{1*}^\varphi\colon H_1 \to H_2$.  Moreover, left-multiplication (in the Morse homology ring) by $\frac{e_1}{r_1}$ is injective on this basis, since $d_{1*}^\varphi \circ (\frac{e_1}{r_1} \cdot)\colon B \to B$ is the identity on $B$, as a quick check shows (see also \cite[Lemma 4.0.6]{Cha:torsion}).  Therefore, we have linearly independent vectors given by
\begin{align*}
\frac{e_1}{r_1} \cdot (e_i - \frac{r_i}{r_1} e_1) & = \frac{e_1 \cdot e_i}{r_1} \quad (\text{since }e_1^2 = 0, \text{ as char } \F \neq 2)\\
& = \frac{1}{r_1}\sum_j I_\F(e_1, e_i, e_j) \bar{e}_j \quad \text{by equation } \eqref{eq:formulaI}
\end{align*}
and $I_\F(e_1, \cdot, \cdot)$ is a non-degenerate two-form on $H_2 / \langle e_1 \rangle$.
\end{proof}
A couple of direct corollaries are worth mentionning.
\begin{cor}
Suppose that the second possibility in the theorem occurs.  Then the dimension of $H_2(L; \F)$ is odd.   Moreover, $L$ is $\F$-prime, i.e.\ if $L=A \# B$, then either $A$ or $B$ is an $\F$-homology sphere.
\end{cor}

In Proposition \ref{prop:detQ}, we prove that whenever the first possibility in the theorem occurs, then $b_2(L;\F)$ is even.  Combined with the proof of Theorem \ref{thm:lag3form}, this yields:
\begin{cor}
Suppose that there exists one narrow representation.  Then the collapsing page of $E^{*, \varphi}_{*,*}$ does not depend on $\varphi \in \mathcal{N}(L;\F)$, it depends only on the form $I_\F$.
\end{cor}

\begin{prop}\label{prop:ringgeneration}
If there exists a vector $e_i$ such that $I_\F(e_i, \cdot, \cdot)$ is symplectic on $H_2 / \langle e_i \rangle$ and $b_2 \geq 2$, then $H_*(L;\F)$ is generated as a ring by $H_2$.  Conversely, if $H_*(L;\F)$ is generated by $H_2$, then there exists a vector $e_i$ such that $I_\F(e_i, \cdot, \cdot)$ is symplectic on $H_2 / \langle e_i \rangle$.
\end{prop}
\noindent The assumption on the Betti number is necessary, as seen by considering $S^1 \times S^2$.
\begin{proof}  Linear algebra.  
\end{proof}

\begin{lem}\label{lem:d1dual}
Given any monotone Lagrangian $L$ of dimension $n$, with minimal Maslov number $N_L=2$, the maps $d_{1*}^\varphi \colon H_0(L;\F) \to H_1(L;\F)$ and $d_{1*}^\varphi \colon H_{n-1}(L;\F) \to H_{n}(L;\F)$ are dual to each other.
\end{lem}
\begin{proof}
Let $\{e_i \}$ be a basis of $H_{n-1}$ and $\{ \bar{e}_i \}$ denote the corresponding Poincar\'e dual basis of $H_1$.  By the divisor axiom for open Gromov-Witten invariants, we have
$$d_{1*}(e_i) = \sum_{A \in \pi_2(M, L)} \# \mathcal{P}(e_i, L; A) \varphi(A) L = \sum_A m_0(A)\varphi(A) (e_i \cdot [\partial A]) L,$$
where $\mathcal{P}(e_i, L; A)$ is the space of pearls going from $e_i$ to $L$ in the class $A$, $m_0(L, A)$ denotes the oriented number of Maslov index two pseudoholomorphic discs going through a generic point, in the homotopy class $A$ (see also \S \ref{sec:superpotential}) and $[\partial A]$ denotes the homotopy class of the boundary of $A$.  Moreover, similar considerations yield
$$d_{1*}(p) = \sum_{i} \sum_{A \in \pi_2(M, L)} \# \mathcal{P}(p, \bar{e}_i; A) \varphi(A) \bar{e}_i = \sum_i \sum_A m_0(A) \varphi(A)([\partial A]\cdot e_i)\bar{e}_i$$
\end{proof}


\section{Invariance of torsion}\label{sec:torinv}
Throughout this whole section, we let $\F$ be a field such that $\text{char } \F$ does not divide $|\tor(L)|$ and is different than two.  This implies that the Betti numbers of $L$ over $\F$ and $\Z$ coincide.  We also fix a generic pearl data triple $\mathcal{D} = (f, \rho, J)$ on $L$ and let $\varphi\colon \pi_2(M, L) \to \F^\times$ be a narrow representation.
\subsection{The second page collapses}
This was essentially proven in \cite{Cha:torsion}.  Using the low dimension of $L$, we make the formulas appearing there more explicit.  First suppose that the narrow representation $\varphi$ is such that the degree spectral sequence collapses at the second page. We use the notations from \S \ref{sec:ring}.  By the proof of Theorem \ref{thm:lag3form}, the first page of $\{E^{*, \varphi}_{*,*}; \F \}$ yields an acyclic chain complex
$$
\xymatrix{0 \ar[r] & H_0(L;\F) \ar[r]^{d^\varphi_{1*}} & H_{1}(L;\F)  \ar[r]^{d^\varphi_{1*}} & H_{2}(L;\F) \ar[r]^{d^\varphi_{1*}} &  H_{3}(L;\F) \ar[r] & 0
}
$$
We denote the torsion of this chain complex, with respect to the basis $h_* \otimes 1$ (see \S \ref{sec:setting} and Lemma \ref{lem:zbasisf}) by $\tau_\varphi(H_*, d^\varphi_{1*}, h_*\otimes 1)$.  It does not depend on the generic choice of pearl data triple $\mathcal{D}$, see \cite{Cha:torsion}.  From \cite[Theorem 4.0.4]{Cha:torsion}, we have $\tau_\varphi(L, \mathcal{D}) = \tau_\varphi(L, h_*\otimes 1) = \dfrac{|\text{Tor}_{\text{ev}}(L)|}{|\text{Tor}_{\text{odd}}(L)|}\tau_\varphi(H_*, d_{1*}^\varphi, h_*\otimes 1)$.  In order to compute the torsion of the first page of the degree spectral sequence, we set $B_i = \text{Im} \{ d^\varphi_{1*}\colon H_{i-1} \to H_i \}$.  We now find explicit bases $b_i$ for these vector spaces.  By assumption, the map $d^\varphi_{1*}\colon H_2 \to H_3$ is surjective.  Set $b_3 = [L]$ and $d^\varphi_{1*}(e_i)= r_i [L]$.  Without loss of generality, we can suppose that $r_1 \neq 0$.  By exactness, $b_2 = \{e_i - \dfrac{r_i e_1}{r_1} \; | \; i \geq 2 \}$ is a basis of $B_2$.  By Lemma \ref{lem:d1dual}, $d^\varphi_{1*}(p) = \sum_i r_i \bar{e}_i$, so we set $b_1 = \sum_i r_i \bar{e}_i$.  Finally, $b_0 = 0$.

Now, one checks that the map $\sigma\colon b_{i+1}\to H_i, \; v \mapsto \frac{e_1}{r_1}\cdot v$, gives a section of the exact sequence $\xymatrix{0 \ar[r] & b_i \ar[r] & H_i \ar[r]^{d^\varphi_{1*}} & b_{i+1} \ar[r]& 0}$.  Note that, by formula (\ref{eq:formulaI}), we have $\sigma \cdot b_2 = \frac{e_1}{r_1} \cdot (e_i - \frac{r_i e_1}{r_1}) =  \frac{e_1}{r_1} \cdot e_i = \frac{1}{r_1} \sum_j I_\F(e_1, e_i, e_j)\bar{e}_j$.  With these choices, the change of basis matrices $[b_i \sigma(b_{i+1}) / h_i]$ are as follows, up to a permutation of the columns:
$$[b_0 \sigma(b_1) / h_0] = (1) = [b_3 / h_3]$$
$$
\begin{blockarray}{ccc}
 & b_1 &  \sigma(b_2)  \\
\begin{block}{c(c|c)}
  \bar{e}_1           & r_1           & 0  \\
\cline{2-3}
  \bar{e}_2           & r_2           &     \\
  \vdots              & \vdots        &  \frac{1}{r_1}I_\F(e_1, e_i, e_j)   \\
  \bar{e}_{b_2(L;\F)} & r_{b_2(L;\Z)} &     \\
\end{block}
\end{blockarray}= [b_1 \sigma(b_2) / h_1]
$$
$$
\begin{blockarray}{ccccc}
 & \sigma(b_3) & e_2 - \frac{r_2 e_1}{r_1} & \cdots & e_{b_2(L;\Z)} - \frac{r_{b_2(L;\Z)} e_1}{r_1} \\
\begin{block}{c(c|ccc)}
  e_1    & \frac{1}{r_1} & \frac{r_2}{r_1} & \cdots & \frac{r_{b_2(L;\Z)}}{r_1}  \\
\cline{2-5}
  e_2    &  0            &  &   &   \\
  \vdots &  \vdots       &  & I &   \\
  e_{b_2(L;\Z)}   &  0            &  &   &   \\
\end{block}
\end{blockarray}= [b_2 \sigma(b_3) / h_2]
$$
Therefore, $\tau_\varphi(H_*, d_{1*}^\varphi) = \dfrac{\det [b_0 \sigma(b_1) / h_0] \det [b_2 \sigma(b_3) / h_2]}{\det [b_1 \sigma(b_2) / h_1] \det [b_3 / h_3]} = \pm \dfrac{r_1^{b_2(L;\Z)-3}}{\det I_\F(e_1, \cdot, \cdot)|_{H_2/ \langle e_1 \rangle}}, \quad$ and
\begin{equation}\label{eq:torsionE1}
 \tau_\varphi(L, h_*\otimes 1) = \dfrac{|\text{Tor}_{\text{ev}}(L)|}{|\text{Tor}_{\text{odd}}(L)|}\cdot \dfrac{r_1^{b_2(L;\Z)-3}}{\det I_\F(e_1, \cdot, \cdot)|_{H_2/ \langle e_1 \rangle}} \in \F^\times / \pm 1.
\end{equation}
Note that $\det I_\F(e_1, \cdot, \cdot)|_{H_2/ \langle e_1 \rangle}$ is not necessarily equal to $\pm 1$; examples are easily given using Sullivan's result from Remark \ref{rem:Sullivan}.

\subsection{The third page collapses}\label{sec:torinv3rdpage}
Suppose that the narrow representation $\varphi$ is such that the associated degree spectral sequence $\{E^{*, \varphi}_{*,*}; \F \}$ collapses at the third page.  By the proof of Theorem \ref{thm:lag3form}, on the first page of $\{E^{*, \varphi}_{*,*}; \F \}$, we have a chain complex
$$
\xymatrix{0 \ar[r] & E^{1, \varphi}_{0,0} \ar@{=}[d] \ar[r]^0 & E^{1, \varphi}_{-1,0}  \ar@{=}[d] \ar[r]^\cong & E^{1, \varphi}_{-2,0} \ar@{=}[d] \ar[r]^0 & E^{1, \varphi}_{-3,0} \ar@{=}[d] \ar[r] & 0\\
0 \ar[r] & H_0^f(L;\F)\otimes P_0 \ar[r]^0 & H_{1}^f(L;\F)\otimes P_{-2}  \ar[r]^{A_{\mathcal{D}}^\varphi \otimes t} & H_{2}^f(L;\F)\otimes P_{-4} \ar[r]^0 &  H_{3}^f(L;\F)\otimes P_{-6} \ar[r] & 0
}
$$
where $A_{\mathcal{D}}^\varphi\colon H_1^f(L;\F) \to H_2^f(L;\F)$ is an isomorphism.  On the second page, we have
$$\xymatrix{
0 \ar[r] & E^{2, \varphi}_{0,0} \ar@{=}[d] \ar[r]^\cong & E^{2, \varphi}_{-2,1}  \ar@{=}[d] \ar[r] & 0\\
0 \ar[r] & H_0^f(L;\F)\otimes P_0 \ar[r]^{r_{\mathcal{D}}^\varphi\otimes t^2} & H_{3}^f(L;\F)\otimes P_{-4}  \ar[r] & 0
}$$
where $r_{\mathcal{D}}^\varphi\colon \F \cong H_0^f(L;\F) \to H_3^f(L;\F) \cong \F$ is an isomorphism.

By assumption on $\text{char } \F$, and by Lemma \ref{lem:zbasisf}, a basis of $H_k^f(L;\F)$ is given by tensoring a basis of the free part of $H_k^f(L;\Z)$ with $\F$, which we denote by $h_k^f \otimes 1$.  Recall the construction of the minimal pearl complex associated to a fixed pearl triple $\mathcal{D}^\varphi_m = (f_m, \rho_m, J_m)$ from \S \ref{sec:minimal}.
\begin{lem}\label{lem:specseqdiff}
With respect to the bases $h_k^f \otimes 1$ and $h_k^{f_m} \otimes 1$, we have $\det A_{\mathcal{D}}^\varphi = \pm \det A_{\mathcal{D}_m}^\varphi$ and $r_{\mathcal{D}}^\varphi = \pm r_{\mathcal{D}_m}^\varphi$.
\end{lem}
\begin{proof}
The quasi-isomorphism $\phi_f\colon \mathcal{C}^\varphi(\mathcal{D}) \to \mathcal{C}^\varphi_{\text{min}}$  from \S \ref{sec:minimal} induces a spectral sequence morphism which is an isomorphism from the first page onwards, see diagram (\ref{eq:diagmin}).  Moreover, it coincides on the first page with the comparison isomorphism in Morse homology, $[\Phi_f^{f_m}]\colon H_*^f(L;\F) \to H_*^{f_m}(L;\F)$.  By assumption on the field, we have $H_*(L;\F) = H_*(L;\Z)\otimes \F$ and the Morse isomorphism is compatible with this identification, i.e.\ $[\Phi_f^{f_m}]_\F = [\Phi_f^{f_m}]_\Z\otimes 1$, where $[\Phi_f^{f_m}]_\Z\colon H_*^f(L;\Z) \to H_*^{f_m}(L;\Z)$.  Therefore, $A_{\mathcal{D}}^\varphi = ([\Phi_f^{f_m}]_\Z\otimes 1)^{-1} A_{\mathcal{D}_m}^\varphi ([\Phi_f^{f_m}]_\Z\otimes 1)$.  The same holds for $r_{\mathcal{D}}^\varphi$ and $r_{\mathcal{D}_m}^\varphi$.  But $ \det ([\Phi_f^{f_m}]_\Z\otimes 1) = (\det [\Phi_f^{f_m}]_\Z)\otimes 1 = \pm 1$.
\end{proof}

In order to lighten the notation, we define $A_\varphi:=A^\varphi_{\mathcal{D}_m}$ and $r_\varphi:= r^\varphi_{\mathcal{D}_m}$.
\begin{thm}\label{thm:torsionformula}
Assume that $\varphi\colon \pi_2(M, L) \to \F^\times$ is a narrow representation such that $E^{*, \varphi}$ collapses at the third page.  Then $\tau_\varphi(L, \mathcal{D}) = \tau_\varphi(L, h_*\otimes 1)= \dfrac{|\tor_{\text{ev}}(L)|}{|\tor_{\text{odd}}(L)|}\cdot \dfrac{\det A_\varphi}{r_\varphi} \in \F^\times/ \pm 1$.  In particular, the torsion does not depend on a choice of generic pearl triple and can be computed using the minimal pearl complex.
\end{thm}
\begin{proof}
First, it is convenient to change the basis of the pearl complex $\mathcal{C}_{[*]}(\mathcal{D})$ from $\text{Crit}_{[*]} f$ to the basis $h_{[*]}b_{[*]}^Mb_{[*-1]}^M$, see \S \ref{sec:milnorsdef}.  We then have, by formulae (\ref{eq:torsionchangebasis}) and (\ref{eq:torsioneqtorsion}),
$$\tau_\varphi(\mathcal{C}_{[*]}, h_{[*]}b_{[*]}^Mb_{[*-1]}^M) = \dfrac{|\tor_\text{odd}(L)|}{|\tor_\text{ev}(L)|} \tau_\varphi(L, \mathcal{D})$$
With respect to these bases, the pearl differentials are given by the following matrices:
\begin{align*}
\begin{blockarray}{c cc  ccc}
 & H_0 & B_0^M & H_2 & B_2^M & s(B_1^M)\\
\begin{block}{c (cc | ccc )}
   H_1         & d_{1*}^\varphi & 0 & 0 & 0 & 0 \\
   B_1^M       & M_1 & M_2 & 0 & 0 & I\\
   s(B_0^M)    & 0 & 0 & 0 & 0 & 0\\
\cline{2-6}
   H_3         & \alpha & M_3 & d_{1*}^\varphi  & 0 & M_6\\
   s(B_2^M)    &  M_4 & M_5 & 0 & 0 & M_7\\
\end{block}
\end{blockarray}
=d^\varphi\colon \Ccal_{[1]}^\varphi \to \Ccal_{[0]}^\varphi \\
\begin{blockarray}{c ccc  cc}
 & H_1 & B_1^M & s(B_0^M) & H_3 & s(B_2^M)\\
\begin{block}{c (ccc | cc )}
   H_0         & 0 & 0 & 0 & 0 & 0 \\
   B_0^M       & 0 & 0 & I & 0 & 0\\
   \cline{2-6}
   H_2    & d_{1*}^\varphi & 0 & V_1 & 0 & 0\\
   B_2^M         & V_2 & V_3 & V_4 & 0 & I\\
   s(B_1^M)    &  0 & 0 & V_5 & 0 & 0\\
\end{block}
\end{blockarray}
=d^\varphi\colon \Ccal_{[0]}^\varphi\to \Ccal_{[1]}^\varphi
\end{align*}
By assumption on $\varphi$, $d_{1*}^\varphi$ is equal to $0$ (respectively $A_\mathcal{D}^\varphi$, $0$) on $H_0$ (resp.\ $H_1$, $H_2$).  Using $d^\varphi \circ d^\varphi =0$, the matrices simplify to
\begin{align}\label{eq:matricesd}
\begin{blockarray}{cc  ccc}
\begin{block}{(cc | ccc )}
   0 & 0 & 0 & 0 & 0 \\
   M_1 & -V_5 & 0 & 0 & I\\
   0 & 0 & 0 & 0 & 0\\
\cline{1-5}
   \alpha & -M_6V_5 & 0 & 0 & M_6\\
   -V_3M_1& V_3V_5 & 0 & 0 & -V_3\\
\end{block}
\end{blockarray} & &
\begin{blockarray}{ccc  cc}
\begin{block}{(ccc | cc )}
   0 & 0 & 0 & 0 & 0 \\
   0 & 0 & I & 0 & 0\\
   \cline{1-5}
   A_\mathcal{D}^\varphi & 0 & V_1 & 0 & 0\\
   V_2 & V_3 & V_4 & 0 & I\\
   0 & 0 & V_5 & 0 & 0\\
\end{block}
\end{blockarray}
\end{align}
Since $H_3 \subset \ker d^\varphi = \text{im } d^\varphi$ (as $L$ is $\varphi$-narrow), we must have $\alpha - M_6M_1 \neq 0$.  Bases for the pearl boundaries are given by
$$
b_{[0]} = \langle d^\varphi(H_1), d^\varphi(s(B_0^M)), d^\varphi(s(B_2^M))\rangle, \quad b_{[1]} = \langle d^\varphi(H_0), d^\varphi(s(B_1^M))\rangle
$$
Moreover, the maps $d^\varphi(H_i) \mapsto H_i, \; d^\varphi(s(B_i^M)) \mapsto s(B_i^M),$ define sections $S\colon b_{[*-1]} \to \Ccal_{[*]}$ of the sequences $$\xymatrix{0 \ar[r]& b_{[*]} \ar[r]& \Ccal_{[*]} \ar[r]^{d^\varphi} & b_{[*-1]} \ar[r]& 0}$$(see \S \ref{sec:torsion})
The associated change of bases matrices $[b_{[*]} S(b_{[*-1]}) / h_{[*]}b_{[*]}b_{[*-1]}]$, up to a reordering of the columns, are then
\begin{align*}
\begin{blockarray}{cccccc}
 & S(d^\varphi(H_0)) & d^\varphi(s B_0^M ) & d^\varphi(H_1) & d^\varphi(s B_2^M) & S(d^\varphi s B_1^M)\\
\begin{block}{c(ccccc )}
   H_0     & I & 0   & 0                     & 0 & 0 \\
   B_0^M   & 0 & I   & 0                     & 0 & 0\\
   H_2     & 0 & V_1 & A_\mathcal{D}^\varphi & 0 & 0\\
   B_2^M   & 0 & V_4 & V_2                   & I & 0\\
   s B_1^M & 0 & V_5 & 0                     & 0 & I\\
\end{block}
\end{blockarray}
= [b_{[0]} S(b_{[1]}) / h_{[0]}b_{[0]}b_{[1]}]\\
\begin{blockarray}{cccccc}
 & S(d^\varphi(H_1)) & d^\varphi(s B_1^M ) & S(d^\varphi sB_0^M) & d^\varphi(H_0) & S(d^\varphi s B_2^M)\\
\begin{block}{c(ccccc )}
   H_1     & I & 0    & 0 & 0       & 0 \\
   B_1^M   & 0 & I    & 0 & M_1     & 0\\
   sB_0^M  & 0 & 0    & I & 0       & 0\\
   H_3     & 0 & M_6  & 0 & \alpha  & 0\\
   s B_2^M & 0 & -V_3 & 0 & -V_3M_1 & I\\
\end{block}
\end{blockarray}
= [b_{[1]} S(b_{[0]}) / h_{[1]}b_{[1]}b_{[0]}]
\end{align*}
from which we get
$$\det [b_{[0]} S(b_{[1]}) / h_{[0]}b_{[0]}b_{[1]}] = \pm \det A_\mathcal{D}^\varphi$$
$$\det [b_{[1]} S(b_{[0]}) / h_{[1]}b_{[1]}b_{[0]}] = \pm(\alpha - M_6M_1)$$
The torsion is  
$$\tau_\varphi(\mathcal{C}_{[*]}, h_{[*]}b_{[*]}^Mb_{[*-1]}^M) = \dfrac{\det A_\mathcal{D}^\varphi}{\alpha - M_6M_1} \in \F /\pm 1$$

\textbf{Last step}: $(\alpha - M_6M_1) = r_\mathcal{D}^\varphi$.
Recall that the only non-trivial differential on the second page of the degree spectral sequence is
$r_\mathcal{D}^\varphi\colon H_0(L;\F) \to H_3(L;\F)$.  However, this map is defined by using an implicit identification $E^{2, \varphi} \cong H(E^{1, \varphi}, d^{1, \varphi})$.  Instead, let us compute the differential on the second page by using the definition of the spectral sequence coming from the degree filtration on the pearl complex.  To do this, we closely follow the book from McCleary \cite[Proof of Theorem 2.6]{McC-Spectral}. Set
\begin{itemize}
\item $F_p \Ccal_{p+q} = \oplus_{j \leq p} \Ccal_{p+q-j N_L} \otimes P_{jN_L}, \quad P_{jN_L} = \F t^{-j}$
\item $\dots \subset F_p\Ccal \subset F_{p+1}\Ccal \subset \dots$
\item $Z^r_{p,q}= \{x \in F_p\Ccal_{p+q} \; | \; d^\varphi(x) \in F_{p-r}\Ccal_{p+q-1}\}$
\item $B^r_{p,q} = F_p \Ccal_{p+q} \cap d^\varphi(F_{p+r}\Ccal_{p+q+1})$
\item $E^r_{p,q} = \dfrac{Z^r_{p, q}}{Z^{r-1}_{p-1, q} + B^{r-1}_{p,q}}$
\end{itemize}
Using matrix (\ref{eq:matricesd}), we get:
$$Z^2_{0,0} = ((H_2 \oplus B_2^M)\otimes P_{-2}) \oplus (S\otimes P_0)$$
where $S$ is the subvector space of  $H_0 \oplus B_0^M\oplus sB_1^M$ given by $\{ (h_0, b_0, V_5(b_0)- M_1(h_0)) \; | \; h_0 \in H_0, b_0 \in B_0^M \}$.  Note that $M_1(h_0)$ is technically speaking an element of $B_1^M$, however $sB_1^M$ is identified with $B_1^M$, it thus makes sense to write $M_1(h_0) \in sB_1^M$.
$$Z^1_{-1, 1}=(H_2\oplus B_2^M)\otimes P_{-2}$$
$$B^1_{0, 0}=(B_2^M \otimes P_{-2})\oplus d^\varphi(H_1 \oplus sB_0^M)$$
Therefore,
$$E_{0, 0}^2 = \{ h_0 - M_1(h_0) \; | \; h_0 \in H_0 \}\otimes P_0 \cong H_0(L;\F)$$
Moreover, $Z_{-2, 1}^2 = H_3\otimes P_{-4}, \quad Z^1_{-3, 2} = 0, \quad B_{-2, 1}^1 = 0$, so that $E_{-2,1}^2 = H_3\otimes P_{-4}$.
The differential on the second page is induced by $d^\varphi$, hence
$$d^{2, \varphi}\colon E_{0, 0}^2 \to E_{-2,1}^2$$
$$h_0 - M_1(h_0) \mapsto (\alpha(h_0) - M_6 M_1(h_0))\otimes t^2 = r_\mathcal{D}^\varphi \otimes t^2$$
By Lemma \ref{lem:specseqdiff}, $r_\mathcal{D}^\varphi$ and $\det A_\mathcal{D}^\varphi$ depend only on $\varphi$, not on the choice of data $\mathcal{D}$.  We conclude that
\begin{equation*}\label{eq:torsionpage3}
\tau_\varphi(\mathcal{C}_{[*]}, h_{[*]}b_{[*]}^Mb_{[*-1]}^M) = \dfrac{\det A_\mathcal{D}^\varphi}{r_\mathcal{D}^\varphi} = \dfrac{\det A_\varphi}{r_\varphi} \in \F /\pm 1
\end{equation*}
\end{proof}

\begin{rem}
The formula $\tau_\varphi(L, h_*\otimes 1)= \dfrac{|\tor_{\text{ev}}(L)|}{|\tor_{\text{odd}}(L)|}\cdot \dfrac{\det A_\varphi}{r_\varphi}$ can be written more succintly as a product of Milnor's torsion (see \S \ref{sec:milnorsdef}) of each page of the spectral sequence $E^{*, \varphi}$, relative to the bases $h_* \otimes 1$ of $H_*(L;\F)$, i.e.\ 
$$\tau_\varphi(L, h_*\otimes 1) = \tau(E^{0, \varphi}, h_*)\tau(E^{1, \varphi}, h_*)\tau(E^{2, \varphi}, h_*)$$  Indeed, the torsion of the complex $(E^{0, \varphi}, d^{0, \varphi})$ is nothing but formula (\ref{eq:torsioneqtorsion}).  The torsion of $(E^{1, \varphi}, d^{1, \varphi})$ is $\det A_\varphi$ and the one of $(E^{2, \varphi}, d^{2, \varphi})$ is $1/r_\varphi$.  It seems that this formula should generalize to spectral sequences that collapse at higher pages, however the choice of basis for $E^{k+1} \cong H_*(E^k)$ is a bit delicate.  In our specific case, the bases are always $h_*$.  In general, the presence of such a canonical choice seems unlikely.
\end{rem}
\subsection{Bifurcation and spectral sequences}
In order to prove invariance of torsion in the previous section, we showed the key fact that $\alpha - M_6 M_1 = r_\varphi$ by looking closely at the degree spectral sequence.  However, the usual way to prove invariance is by studying the possible bifurcations that can happen in a generic one parameter family of complexes.  We explain here why the spectral sequence argument actually encodes all significant bifurcations.  Let us stress that we do not state any precise theorems, nor do we provide proofs.  A precise statement relating bifurcations and spectral sequences remains to be found.  The discussion here was inspired by an article from Hutchings \cite[\S 3.6]{Hut:torsion}.

For simplicity, assume that $L$ admits a perfect Morse function $f$, with a minimum $p$ and a maximum $[L]$.  Then, the matrices $M_1^f, M_6^f$ above are equal to zero and the coefficient $\alpha_f$ is given by a count of pearls from $p$ to $[L]$.  Assume that $f$ is now changed to a Morse-Smale function $h$ by adding a single pair of critical points $x, y$, such that $|x|=1, \; |y|=2$, and such that there is a unique negative gradient flow line $\gamma$ from $y$ to $x$.  In other words, $h$ is obtained from $f$ by a single birth of critical points.  For this $h$, the coefficients $\alpha_h, M_1^h, M_6^h$ may be different than the ones for $f$.  By definition, $\alpha_h - M_6^h M_1^h$ counts points in the space $\mathcal{P}(p, x; h) \times \gamma \times \mathcal{P}(y, [L]; h)$, where $\mathcal{P}(z,w; g)$ is the space of pearls from $z$ to $w$ for a function $g$.  Performing a death of $x, y$, i.e.\ slowly shrinking $\gamma$ to a degenerate critical point $z_0$, then removing that point, we get a one parameter family of functions $f_t$ and a  cobordism between $\mathcal{P}(m,[L]; f)$ and $\mathcal{P}(p, x; h) \times \gamma \times \mathcal{P}(y, [L]; h)$.  Therefore, the counts on both sides of this cobordism yield the same number, given by $r_\varphi$.
\begin{remnonum}
The main technical difficulty required to make this discussion precise is the gluing of a pearl at a degenerate critical point.  Moreover, a complete description of all possible generic bifurcations in the space of pearls of a Lagrangian 3-fold should be given, as well as a comparison of the pearl differentials on each side of these bifurcations.
\end{remnonum}

\section{Torsion and quantum product}\label{sec:torsionandqproduct}
\subsection{The quantum product}
In this section, we express the coefficient $r_\varphi$ from Theorem \ref{thm:torsionformula} in terms of the chain-level quantum product on the minimal pearl complex.  Under the assumptions of Theorem \ref{thm:torsionformula}, for any two classes $x, y \in H_2(L;\F)$, we have, by Theorem \ref{thm:lag3form}, $x \cdot y =0$.  Therefore, the quantum product on the minimal pearl complex $(H^{f_m}_*(L;\F), \delta^\varphi)$ defines a bilinear functional
\begin{align}\label{eq:biform}
Q_\varphi\colon H_2(L;\F) \otimes H_2(L;\F) & \to H_3(L;\F)\cong \F\\
x \otimes y & \mapsto x \circ y \nonumber
\end{align}
Recall that it depends on a choice of auxiliary Morse functions $f_m, f_1, f_2, f_3$, see \S \ref{sec:minimal}.  Compare with the quadratic form defined by Biran-Cornea in \cite[\S 4.1]{Bi-Co:lagtop}.

In order to compute the morphism $r_\varphi\colon H_0(L;\F) \to H_3(L;\F)$, we pick Poincar\'e dual bases of the free part of $H_1(L;\Z)$ and $H_2(L;\Z)$ and tensor them with $1_\F$.  Denote by $\bar{e}_i$ the associated basis vectors of $H_1(L;\F)$ and by $e_j$ the basis of $H_2(L;\F)$, as was done in \S \ref{sec:ring}.  On the chain level, we have
$\bar{e}_i \circ e_j = \bar{e}_i \cdot e_j + h.o.t.$, since the minimal quantum product is a deformation of the intersection product.  Also, $h.o.t.$ is an element of $H_2$, for degree reasons.  By assumption on the narrow representation $\varphi$, $\delta^\varphi|_{H_2}$ vanishes.  Therefore, $\delta^\varphi(\bar{e}_i \circ e_j) = \delta^\varphi(\bar{e}_i \cdot e_j)=\delta_{i,j} r_\varphi$, where $\delta_{i,j}$ is the Kronecker delta. Moreover, by definition of $\delta^\varphi$, we see that $\delta^\varphi|_{H_1} = A_\varphi$.  Hence, we have
\begin{align}\label{eq:formularphi}
\delta_{i,j} r_\varphi &= \delta^\varphi(\bar{e}_i \circ e_j) = \delta^\varphi(\bar{e}_i) \circ e_j + \bar{e}_i\circ \delta^\varphi(e_j)= A_\varphi(\bar{e}_i) \circ e_j = Q_\varphi(A_\varphi(\bar{e}_i), e_j)
\end{align}
On the other hand, $\bar{e}_i \cdot e_j = e_j \cdot \bar{e}_i$, so $\delta_{i, j} r_\varphi = \delta^\varphi(e_j \circ \bar{e}_i) = -Q_\varphi(e_j, A_\varphi(\bar{e}_i))$.  Recall that $A_\varphi\colon H_1 \to H_2$ is an isomorphism.  Therefore, $Q_\varphi$ is in fact antisymmetric.  With respect to the basis $A_\varphi(\bar{e}_i) \otimes e_j$ of $H_2 \otimes H_2$, $Q_\varphi$ is given by a diagonal matrix of dimension $b_2(L;\Z)$ with diagonal entries all equal to $r_\varphi \neq 0$.  Summarizing this discussion, and using Lemma \ref{lem:specseqdiff}, we have:
\begin{prop}\label{prop:detQ}
$\det Q_\varphi$ depends only on $\varphi$, not on the choice of Morse functions used in defining $Q_\varphi$.  Moreover, $Q_\varphi$ is a symplectic form on $H_2(L;\F)$, hence $b_2(L;\F) = b_2(L;\Z)$ is even.  In the basis $e_i\otimes e_j$ of $H_2 \otimes H_2$, $\det Q_\varphi = \dfrac{r_\varphi^{b_2(L;\Z)}}{\det A_\varphi}$.
\end{prop}
By Lemma \ref{lem:zbasisf}, any Poincar\'e dual basis $h_*\otimes 1$ of $H_*(L;\F)$ are equivalent.  From Theorem \ref{thm:torsionformula}, we conclude that for every narrow representation $\varphi$ that vanishes on the third page, the quantum Reidemeister torsion satisfies 
\begin{equation}\label{eq:torsionE2}
\tau_\varphi(L, h_*\otimes 1)^{b_2(L;\Z)}= \left(\dfrac{|\tor_{\text{ev}}(L)|}{|\tor_{\text{odd}}(L)|}\right)^{b_2(L;\Z)} \dfrac{\det A_\varphi^{b_2(L;\Z)-1}}{\det Q_\varphi}
\end{equation}
explaining a formula that was first  observed in a simpler context by the author in \cite[\S 6]{Cha:torsion}.  

\begin{rem}
The coefficient $r_\varphi$ from Theorem \ref{thm:torsionformula} was shown to be equal to $\alpha - M_6M_1 = Q_\varphi(A_\varphi(\bar{e}_i), e_i)$.  Note that the term $\alpha - M_6M_1$ involves, by definition, moduli spaces of pearl trajectories using discs of Maslov index four with two marked points, as well as pearls with two discs of Maslov index two joined by a flow line.  On the other hand, both $Q_\varphi$ and $A_\varphi$ involve only discs of Maslov index two, with either two or three marked points.
\end{rem}

\section{The space of complex representations}\label{sec:complexrep}
\subsection{The superpotential}\label{sec:superpotential}
We momentarily drop the assumption that $L$ is a 3-manifold and suppose only that $L$ is a closed, orientable, spin, and monotone Lagrangian submanifold of $(M, \omega)$.  We let $\dim L = n$.  Our presentation of the superpotential will follow closely Biran-Cornea's approach \cite[\S 2.4 \& 3.3]{Bi-Co:lagtop}.  

Let $\hom_0(H_1; \C^\times)$ be the set of representations of $H_1(L;\Z)$ that are equal to 1 on the torsion subgroup.  The space of wide representations is defined as $\mathcal{W}_1 = \{ \varphi \in \hom_0(H_1; \C^\times) \; | \; QH^\varphi(L) \cong H_*(L;\C)\}$.  This space is in fact an algebraic subvariety of $\hom_0(H_1; \C^\times)$, by \cite[Proposition 3.1.1]{Bi-Co:lagtop}.  Similarly, the space of narrow representations is $\mathcal{N}_1 = \{ \varphi \in \hom_0(H_1; \C^\times) \; | \; QH^\varphi(L)=0 \}$.  The narrow/wide conjecture states that the complement of $\mathcal{W}_1$ is $\mathcal{N}_1$.

For $A \in \pi_2(M, L)$ such that $\mu(A) = 2$, denote by $m_0(L, A)$ the oriented number of pseudoholomorphic discs in the homotopy class $A$ whose boundary goes through a generic point of $L$. By Gromov compactness and monotonicity of $L$, there are only a finite number of classes $A$ for which this number is not zero, and it does not depend on the choice of compatible almost complex structure.  The Landau-Ginzburg superpotential of $L$ is defined by
\begin{equation}\label{eq:superpotential}
\SP\colon \hom_0(H_1; \C^\times) \to \C, \quad \varphi \mapsto \sum_{A \in \pi_2(M, L)} m_0(L, A)\varphi(\partial A)
\end{equation}
where $\partial\colon \pi_2(M, L) \to \pi_1(L)$ is the connecting morphism and $\partial A$ is considered as the image of the Hurewicz morphism $\pi_1(L) \to H_1(L;\Z)$.  Note that $\SP(\varphi)$ is called the obstruction number of $L$ (with respect to $\varphi$) and is sometimes denoted by $m_0(L;\varphi)$.  Pick $\bar{e}_1, \dots, \bar{e}_{b_1}$, a basis of the free part of $H_1(L;\Z)$ and $z_1, \dots, z_{b_1},$ the corresponding dual basis of $\hom_0(H_1; \C^\times) \cong (\C^\times)^{b_1}$, where $b_1 = b_1(L;\Z)$ is the first Betti number.  In those coordinates, the superpotential is a Laurent polynomial given by
$$\SP(z_1, \dots, z_{b_1}) = \sum_{A \in \pi_2(M, L)}m_0(L, A)z_1^{(\partial A)_1} \cdots z_{b_1}^{(\partial A)_{b_1}}$$
where $\partial A = \sum_i \bar{e}_i (\partial A)_i$.  Pick a basis $e_1, \dots, e_{b_{n-1}}$ of $H_2(L;\Z)/\tor$, Poincar\'e dual to $\bar{e}_i$.  Let $\Ccal_{min}^\varphi = (H_*(L;\C), \delta^\varphi)$ be the minimal pearl complex and consider the chain-level quantum product $\circ$ on this complex, restricted to $H_{n-1}$.
\begin{prop}[Propositions 3.3.1 and 3.3.4 in \cite{Bi-Co:lagtop}]\label{prop:BiCo} For $\varphi = (z_1, \dots, z_{b_1})$, we have
\begin{itemize}
 \item $\delta_{1}^\varphi(e_i) = z_i \dfrac{\partial \SP}{\partial z_i}(z_1, \dots, z_{b_1})[L]$.
 \item If $QH^\varphi(L;\C) \neq 0$, then $\varphi$ is a critical point of $\SP$.
 \item For $\varphi \in \Crit \SP$,  $e_i \circ e_j + e_j \circ e_i = (-1)^n z_i z_j \dfrac{\partial \SP}{\partial z_i \partial z_j}(z_1, \dots, z_{b_1})[L]$
\end{itemize}
In particular, $\mathcal{W}_1 \subset \Crit \SP$.  Moreover, if $H_*(L;\R)$ is generated as a ring by $H_{n-1}(L;\R)$, then $\mathcal{W}_1 = \Crit \SP$ and the narrow-wide conjecture is true.
\end{prop}
Associated to a Lagrangian of minimal Maslov class two, there is a discriminant function $\Delta\colon \mathcal{W}_1 \to \C$ (see \cite[\S 4.3]{Bi-Co:lagtop}) given by 
\begin{equation}\label{eq:discriminant}
\Delta(z_1, \dots, z_{b_1}) = (-1)^{nb_1 + 1}z_1^2 \cdots z_{b_1}^2 \det \left( \dfrac{\partial^2 \SP}{\partial z_i \partial z_j} \right)
\end{equation}
Finally, we recall a theorem of Auroux, Kontsevich, and Seidel (see Auroux \cite[\S 6]{Aur:t-duality} and Sheridan \cite[\S 2.9]{She:Fano} for details).  Assume that $(M, \omega)$ is a closed symplectic manifold and denote by $c_1$ the first Chern class of $M$.  Consider the quantum cohomology $QH(M, \omega)$, endowed with the quantum product $\star$, as a complex vector space of dimension $D = \sum_0^{2n} b_i(M;\C)$ and denote by $\lambda_1, \dots, \lambda_{D}$ the generalized eigenvalues of the endomorphism $c_1 \star \colon QH(M, \omega) \to QH(M, \omega)$.  Then
\begin{thm}[Auroux, Kontsevich, Seidel]
 Consider a representation $\varphi = (z_1, \dots, z_{b_1}) \in \hom_0(H_1; \C^\times)$.  If $QH^\varphi(L) \neq 0$, then there exists $i$ such that $\SP(\varphi) = \lambda_i$.  In particular, $\SP(\mathcal{W}_1) \subset \SP(\hom_0(H_1; \C^\times) \backslash \mathcal{N}_1) \subset \{ \lambda_1, \dots, \lambda_{D} \}$.
\end{thm}
Since $\SP$ is a holomorphic function, we get:
\begin{cor}\label{cor:narrowexist}
 If $\SP$ is not constant, then there exist narrow representations.
\end{cor}

\subsection{Applications to Lagrangian 3-manifolds: Proof of Theorem \ref{thm:main}}\label{sec:mainproof}
We now come back to monotone, orientable, closed Lagrangian 3-manifolds.
\begin{proof}
By the proof of Theorem \ref{thm:lag3form} and Proposition \ref{prop:detQ}, the collapsing page of the degree spectral sequence depends only on the parity of $b_2(L;\Z)$, not on the actual narrow representation, and it restricts the intersection 3-form $I_\C$ as desired.  

If $b_2$ is odd, the second page collapses, which means, by Proposition \ref{prop:BiCo}, that $\varphi \not\in \Crit \SP$.  Then, torsion is given by formula (\ref{eq:torsionE1}).  If $b_2 \geq 3$, then Propositions \ref{prop:ringgeneration} and \ref{prop:BiCo} imply that $\mathcal{W}_1 = \Crit \SP$.

As for the second point,  since $L$ is $\varphi$-narrow and $b_2(L;\Z)$ is even, the only possibility is that the degree spectral sequence collapses at the third page.  By definition, this means that $H_*(E^{1, \varphi}) \neq 0$, so that $\varphi \in \Crit \SP$ and $\hom_0(H_1; \C^\times) \backslash \Crit \SP = \emptyset$, by the first point of Proposition \ref{prop:BiCo}.  Therefore, $\SP$ is constant.  Formulae (\ref{eq:torsionE2}) and (\ref{eq:discriminant}) conclude the proof.
\end{proof}


\bibliographystyle{alpha}
\bibliography{../../../biblio_latex/bibliography}
\end{document}